\documentclass[12pt]{article}

\def\UseSection{
        \numberwithin{equation}{section}
	\theoremstyle{plain}
        \newtheorem{theorem}    {Theorem}[section]
        \DefineTheorems 
}

\usepackage{amsfonts}
\usepackage{amsmath,amssymb,amsthm}
\usepackage{appendix}
\usepackage{bbm} 
\usepackage{amsbsy}
\usepackage{enumerate}
\usepackage{cite}
\usepackage{enumerate}
\usepackage{xr-hyper}
\usepackage{hyperref}
\hypersetup{
    colorlinks,
    citecolor=black,
    filecolor=black,
    linkcolor=black,
    urlcolor=black
}
\usepackage[textwidth=500pt,textheight=650pt,centering]{geometry} 
\usepackage[dvips]{graphicx}

\usepackage{mathrsfs}

\def\DefineTheorems{
	\newtheorem{lemma}      [theorem] {Lemma}
	\newtheorem{cor}        [theorem] {Corollary}
	\theoremstyle{definition}
	
}

\UseSection
\newcommand{\C}{{\mathbb{C}}}
\newcommand{\Ex}{{\mathbb{E}}}
\newcommand{\N}{{\mathbb{N}}}
\newcommand{\R}{{\mathbb{R}}}
\newcommand{\Z}{{\mathbb{Z}}}
\newcommand{\Scal}{{\cal S}}

 \title {
   Self-avoiding walk on the complete graph
 }

 \author{
   Gordon Slade\thanks{Department of Mathematics,
     University of British Columbia,
     Vancouver, BC, Canada V6T 1Z2.
     https://orcid.org/0000-0001-9389-9497.
     E-mail: {{\tt slade@math.ubc.ca}}}}

 \date{\vspace{-5ex}} 

\begin{document}
\maketitle

\begin{abstract}
There is an extensive literature concerning self-avoiding walk on infinite graphs, but the
subject is relatively undeveloped on finite graphs.  The purpose of this paper
is to elucidate the phase transition for self-avoiding walk on the simplest
finite graph:  the complete graph.
We make the elementary observation that the susceptibility of the self-avoiding walk
on the complete graph is given exactly in terms of the incomplete gamma function.
The known asymptotic behaviour of the incomplete gamma function then yields
a complete description of the finite-size scaling of the self-avoiding walk on
the complete graph.
As a basic example, we compute the limiting distribution of the length of a self-avoiding
walk on the complete graph, in subcritical, critical,
and supercritical regimes.
This provides a prototype for more complex unsolved problems such as
the self-avoiding walk on the hypercube or on a high-dimensional torus.
\end{abstract}

%

\section{Introduction and results}

\subsection{Self-avoiding walk}

The self-avoiding walk is a mathematical model of interest in
combinatorics, in probability, in statistical mechanics, and in polymer science
\cite{MS93}.
A basic problem in the subject is to count the number of
fixed-length non-self-intersecting paths
in a transitive graph, starting from some fixed origin.  The most common setting
is to take the graph to be the $d$-dimensional integer lattice $\Z^d$ with
nearest-neighbour edges.  A simple subadditivity argument implies that
the number $c_N$ of self-avoiding walks on $\Z^d$ of length $N$
satisfies $c_N^{1/N} \to \mu$ as $N\to\infty$, and also
$c_N \ge \mu^N$ (for all $N$).  The \emph{connective
constant} $\mu$ obeys $\mu \in [d,2d-1]$, and in particular
depends on the dimension $d$.  The generating function
$\chi_z = \sum_{N=0}^\infty c_Nz^N$ of the sequence $(c_N)$
is called the \emph{susceptibility}, and it has a finite
radius of convergence equal to the \emph{critical value} $z_c=\mu^{-1}$.
For $z\ge z_c$, the susceptibility is infinite and contains no information.
Many of the most important mathematical problems for the self-avoiding walk
remain unsolved, e.g., the precise asymptotic behaviour
on $\Z^2$ or $\Z^3$ of $c_N$ (as $N\to\infty$)
or the susceptibility (as $z\uparrow z_c$).

On a finite graph, $c_N$ is zero once $N$ reaches the cardinality
of the vertex set, so the susceptibility is a polynomial in $z$ and hence
is defined for all $z \in \C$.  For large but finite graphs, it can be expected
that the divergence of the susceptibility in the infinite setting
will be reflected by different
behaviour of the susceptibility above and below some \emph{critical scaling window}
of positive $z$ values---a phase transition.
It is further to be expected that on high-dimensional graphs this
transition should share qualitative features with the phase transition for
the Erd\H{o}s--R\'{e}nyi random graph.

In this paper, we elucidate the phase transition for self-avoiding walk
on the simplest finite
graph, the complete graph.  We observe that the
self-avoiding walk is exactly solvable on the complete graph,
in the sense that the susceptibility can be expressed exactly in terms of
the upper incomplete gamma function.  We use the known asymptotic
behaviour of the incomplete gamma function to derive the asymptotic behaviour
of the susceptibility.  This leads to an identification of
the critical scaling window for the phase
transition, and of the different behaviours that occur below, above, and
within the critical window.  In particular, with respect to
the probability distribution on the set of all self-avoiding walks (of any length)
which assigns probability proportional to $z^N$ to a walk of length $N$,
we identify the
different behaviours of the length of a random self-avoiding walk in the different regimes.

For self-avoiding walk on the complete graph, we show that the phase transition corresponds
exactly to a transition in the asymptotic behaviour of the incomplete gamma function.
Complete information and historical
background for the asymptotic behaviour of the
incomplete gamma function is provided in \cite{NO19},
from which the finite-size scaling of
the susceptibility can be read off.
The new results of \cite{NO19} are precisely in the region of most interest---what
they call the \emph{transition region} and what for us is the critical scaling
window---though
we do not use anything like the full power of the results of \cite{NO19}.
This is an interesting example of a connection between a transition in the
asymptotic behaviour of
a special function and a phase transition in the finite-size scaling
of a statistical mechanical model.

Although the complete graph lacks geometry, it is an example which serves as a
prototype for high-dimensional finite graphs which do have interesting
geometry, such as the hypercube $\{0,1\}^n$ with nearest-neighbour edges, or
high-dimensional discrete tori.  Percolation on the hypercube and high-dimensional
tori has been well-studied \cite{BCHSS05a,HH17book}, and it would be of interest
to develop the theory of self-avoiding walk on these graphs.

\subsection{Exact solution for susceptibility}

Let $K_n$ denote the complete graph on $n$ vertices.
An $N$-step self-avoiding walk on $K_n$ is a sequence of $N+1$ distinct vertices
in $K_n$.
Let $\Scal_N^{(n)}$ denote the set of $N$-step self-avoiding walks on the complete
graph $K_n$, started from some fixed origin.
The cardinality of $\Scal_N^{(n)}$ is $c^{(n)}_0=1$ and
\begin{equation}
    c^{(n)}_N = (n-1)(n-2)\cdots (n-N) = \frac{(n-1)!}{(n-1-N)!}
    \qquad (N=1,\ldots,n-1).
\end{equation}
The \emph{susceptibility} is the generating function for $c^{(n)}_N$, namely the
polynomial
\begin{align}
    \chi^{(n)}_z & = \sum_{N=0}^{n-1} c^{(n)}_N z^N
    =
    (n-1)!z^{n-1} \sum_{m=0}^{n-1} \frac{1}{m!z^m}.
\end{align}
For $x\ge 0$ and $s>0$, the upper incomplete gamma function is defined by
\begin{equation}
\label{e:igf}
    \Gamma(s,x) = \int_x^\infty t^{s-1} e^{-t} dt
    .
\end{equation}
Analytic continuation in both variables $s,x$ is possible.
For $n \in \N$,
\begin{equation}
    \Gamma(n,x) = (n-1)! e^{-x} \sum_{k=0}^{n-1}\frac{x^k}{k!}.
\end{equation}
This gives an exact formula for the susceptibility:
\begin{align}
\label{e:chigam}
    \chi^{(n)}_z & =
    z^{n-1} e^{1/z} \Gamma(n,1/z)
    .
\end{align}

\subsection{Random length}

For $z > 0$, we define a probability measure on $\Scal^{(n)}
= \cup_{N=0}^{n-1} \Scal_N^{(n)}$ by
assigning probability $z^{|\omega|}/\chi^{(n)}_z$ to $\omega \in \Scal^{(n)}$,
where $|\omega|$
denotes the number of steps in $\omega$.  Expectation with respect
to this measure is denoted $\Ex^{(n)}_z$.
Given $z$ and $n$, we define the random variable $L$ to be the length $|\omega|$
of a walk $\omega$ drawn
randomly from $\Scal^{(n)}$.

By definition,  and
by \eqref{e:chigam} and \eqref{e:igf}, the mean of $L$ is
\begin{align}
\label{e:L1}
    \Ex_z^{(n)} (L) & = \frac{1}{\chi^{(n)}_z}\sum_{N=0}^{n-1} Nc_N^{(n)}z^N
    =
    z \frac{d}{dz} \log \chi^{(n)}_z
    =
    (n-1) - \frac{1}{z} + \frac{1}{z\chi^{(n)}_z}.
\end{align}
Higher moments can also be computed exactly.  For example, with $\ell = \Ex_z^{(n)} (L)$
and $\chi = \chi_z^{(n)}$,
\begin{align}
    \Ex_z^{(n)} (L^2) & =
    \frac{1}{\chi} \left( z \frac{d}{dz}\right)  \left(z \frac{d}{dz}\chi \right)
    =
    \frac{1}{\chi} \left( z \frac{d}{dz}\right)
    \left(\chi \ell\right)
    = \ell^2 + z \frac{d}{dz}\ell
    \nonumber \\ & =
    \ell^2
    + \frac{1}{z} \left( 1 - \frac{\ell}{\chi} - \frac{1}{\chi} \right),
\end{align}
and hence the variance of $L$ is
\begin{equation}
\label{e:VarL}
    {\rm Var}_z^{(n)} (L)
    =
    \frac{1}{z}
    \left(
    1 - \frac{1}{\chi_z^{(n)}}(\Ex_z^{(n)} (L)+1)
     \right).
\end{equation}
Let $z_t = ze^t$.
The moment generating function of $L$ is, by definition and by \eqref{e:chigam},
\begin{equation}
\label{e:mgf}
    M_L(t)  = \Ex_z^{(n)} e^{tL} = \frac{1}{\chi_z^{(n)}}\chi_{z_t}^{(n)}
    =
    \frac{z_t^{n-1} e^{1/z_t} \Gamma(n,1/z_t)}{z^{n-1} e^{1/z} \Gamma(n,1/z)}
    \qquad  (t\in\R).
\end{equation}

\subsection{Results}
\label{sec:results}

The asymptotic behaviour of the susceptibility is given by the following theorem.
The gamma function is $\Gamma(s)=\Gamma(s,0)$.
The complementary error function ${\rm erfc}$ in \eqref{e:chicrit} is
\begin{equation}
    {\rm erfc}(x) = 1 - {\rm erf}(x) = \frac{2}{\sqrt{\pi}} \int_x^\infty e^{-t^2}dt
    \qquad  (x \in \R),
\end{equation}
and it extends to an entire analytic function on the complex plane.
We use the notation $f_n \sim g_n$ to mean that $\lim f_n/g_n =1$.

\begin{theorem}
\label{thm:chi}
Let $r> \frac 12$ and $y_n=1/z_n$.  If eventually $y_n \ge n+ n^r$ (subcritical case)
then
\begin{align}
\label{e:chisub1}
     \chi_{z_n}^{(n)} &= \frac{y_n}{y_n-n}
     \left(1-\frac{y_n}{(y_n-n)^2}
     + O\Big( \frac{y_n^2}{(y_n-n)^4} \Big) \right)
     .
\end{align}
Let $\tau_n,\tau\in\R$ with $\tau_n\to\tau$, and let $y_n=n+\tau_n n^{1/2}$ (critical case). Then
\begin{align}
\label{e:chicrit}
    \chi_{z_n}^{(n)} &\sim
     \sqrt{2\pi n}
     \, e^{\tau^2/2} \frac 12  {\rm erfc}(2^{-1/2}\tau)
     .
\end{align}
If eventually $y_n \le n  -n^r$ (supercritical case)
then
\begin{align}
\label{e:chisup1}
    \chi_{z_n}^{(n)} &\sim
     y_n^{1-n}e^{y_n}\Gamma(n).
\end{align}
\end{theorem}

Special cases of the subcritical case
are\footnote{In fact, \eqref{e:chisub} is an immediate consequence of the dominated convergence
theorem, because $\lim_{n\to\infty}\sum_{N=0}^{n-1}n^{-N}c_N^{(n)}s^{-N} = (1-s^{-1})^{-1}$
since $n^{-N}c_N^{(n)} \le 1$ and $\lim_{n\to\infty}n^{-N}c_N^{(n)} = 1$ for each $N$.}
\begin{alignat}{2}
\label{e:chisub}
     \chi_{z}^{(n)} &
     \to \frac{s}{s-1}
     \qquad
     & (z= \tfrac{1}{sn},\, s>1),
    \\
\label{e:chinearsub}
    \chi_{z}^{(n)} &\sim
    \frac{1}{a} n^{1-q}
    \qquad & (z= \tfrac{1}{n+an^q}, \, q \in (\tfrac 12, 1)).
\end{alignat}
By Stirling's formula, special cases of the supercritical case are
\begin{alignat}{2}
\label{e:chinearsup}
    \log \chi_{z}^{(n)} & =
     \frac 12 a^2 n^{2q-1} + \log \sqrt{2\pi n} + O(n^{3q-2}) +o(1)
     \qquad
    & (z= \tfrac{1}{n-an^q}, \, q \in (\tfrac 12, 1)),
    \\
\label{e:chisup}
    \chi_{z}^{(n)} &\sim  s\sqrt{2\pi n} \, e^{(|\log s| +s-1)n}
    \qquad & (z=\tfrac{1}{sn},\, s\in (0,1)).
\end{alignat}
The exponent in \eqref{e:chisup} is positive for $s<1$, so there
is exponential growth.
For $q \in (\frac 12 , \frac 23)$, \eqref{e:chinearsup} yields the asymptotic
relation  $\chi_{z}^{(n)} \sim  \sqrt{2\pi n} \, e^{\frac 12 a^2 n^{2q-1}}$,
whereas for $q\in [\frac 23, 1)$ the error term $O(n^{3q-2})$ for the logarithm is
not insignificant.
It then follows from \eqref{e:VarL},
and since $L<n$ by definition, that ${\rm Var}_z^{(n)}(L)
\sim z^{-1}$ in the supercritical case.

The crossover between the critical and subcritical regimes, and between the critical
and supercritical regimes, is given by an extension of \eqref{e:chicrit}, namely that
as long as $|\tau_n| \le O(n^{q-1/2})$ with $q \in [\frac 12,\frac 23)$,
\begin{align}
\label{e:chicritextension}
    \chi_{z_n}^{(n)} &\sim
     \sqrt{2\pi n}
     \, e^{\tau_n^2/2} \frac 12  {\rm erfc}(2^{-1/2}\tau_n)
     .
\end{align}
This reduces to $\chi_{z}^{(n)} \sim  \sqrt{2\pi n} \, e^{\frac 12 a^2 n^{2q-1}}$
when $\tau_n=-an^{q-1/2}$ (since then
${\rm erfc}(2^{-1/2}\tau_n) \to 2$), and to
\eqref{e:chinearsub} when $\tau_n=an^{q-1/2}$ (since then
${\rm erfc}(2^{-1/2}\tau_n) \sim \sqrt{2/\pi}\tau_n^{-1} e^{-\tau_n^2/2}$).
The proof of \eqref{e:chicritextension} is included in the proof of Theorem~\ref{thm:chi}.

From Theorem~\ref{thm:chi}, it is straightforward to deduce the following
asymptotic behaviour for the expected length, by inserting the formulas for
the susceptibility into \eqref{e:L1}.  The second-order term
in \eqref{e:chisub1} is needed for
the subcritical case of \eqref{e:EL}.
The constant in the critical case of \eqref{e:EL} is defined by
\begin{equation}
    \alpha_\tau = -\tau +\sqrt{2/\pi}e^{-\tau^2/2}[{\rm erfc}(2^{-1/2}\tau)]^{-1},
\end{equation}
which is the strictly positive decreasing function of $\tau\in\R$
plotted in Figure~\ref{fig:alpha}.  It obeys $\alpha_\tau \sim -\tau$ as $\tau \to -\infty$
and $\alpha_\tau \sim \tau^{-1}$ as $\tau \to +\infty$.

\begin{cor}
\label{cor:EL}
Let $r> \frac 12$ and $\tau\in \mathbb{R}$.
With the subcritical, critical, and supercritical cases as specified in
Theorem~\ref{thm:chi},
the expected length has the asymptotic behaviour, as $n \to \infty$,
\begin{align}
\label{e:EL}
     \Ex_{z_n}^{(n)} (L) &\sim
     \begin{cases}
    \frac{n}{\frac{1}{z_n} -n} & (\text{subcritical})
    \\
    \alpha_\tau n^{1/2} \qquad & (\text{critical})
    \\
    n - \frac{1}{z_n} & (\text{supercritical})
    .
    \end{cases}
\end{align}
\end{cor}

\begin{figure}
\centering{
\includegraphics[scale=1.0]{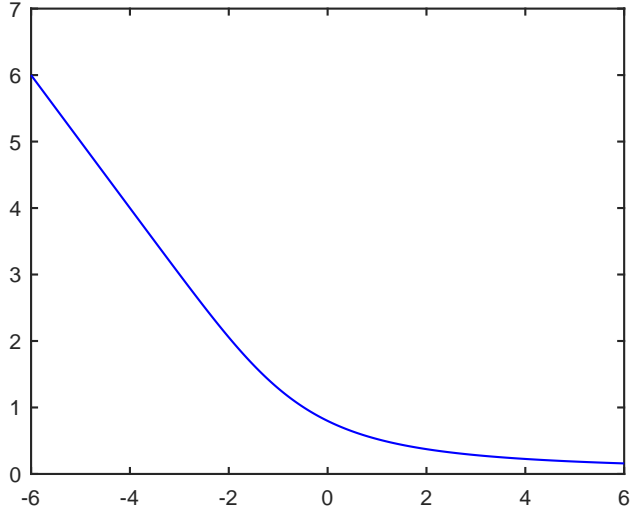}
\caption{Graph of $\alpha_\tau$ vs $\tau$.}
\label{fig:alpha}
}
\end{figure}

\begin{proof}
Let $y=1/z$ and $\delta =y-n$.
For the subcritical case, by \eqref{e:L1} and Theorem~\ref{thm:chi},
\begin{align}
    \Ex_z^{(n)} (L) &
    =
    n-1-y + y \chi^{-1}
    = -1-\delta + \delta \left(1 + \frac{y}{\delta^2}
    + O\Big( \frac{y^2}{\delta^4} \Big) \right)
    \nonumber \\ & =
    \frac{n}{\delta}\left(1  + O\Big( \frac{y^2}{n\delta^2} \Big) \right)
    .
\end{align}
Since $y^2=n^2+2\delta n + \delta^2$,
the error term on the right-hand side is bounded above by a multiple of
$n\delta^{-2} + \delta^{-1} +n^{-1}$ and hence goes to zero as $n\to\infty$ since
eventually $\delta^2 \ge n^{2r}$.  This completes the proof of the subcritical case.
The supercritical case follows from \eqref{e:L1} and the fact that $z\chi \to \infty$
by \eqref{e:chinearsup} with $q=r$ (and monotonicity of $\chi$ in $z$).
The critical case follows from \eqref{e:L1} and \eqref{e:chicrit}.
\end{proof}

Special cases of the subcritical and supercritical behaviour are:
\begin{align}
\label{e:ELsub}
     \Ex_z^{(n)} (L) &\sim
     \begin{cases}
    \frac{1}{s-1} & (z=\tfrac{1}{sn},\, s>1)
    \\
    \frac{1}{a} n^{1-q} & (z= \tfrac{1}{n+an^q}),
    \end{cases}
\end{align}
\begin{align}
\label{e:ELsup}
     \Ex_z^{(n)} (L) &\sim
     \begin{cases}
    an^q & (z= \tfrac{1}{n-an^q})
    \\
     n(1-s) & (z=\tfrac{1}{sn},\, s\in (0,1))
    .
    \end{cases}
\end{align}
The crossover from critical to subcritical or supercritical can be seen from
\eqref{e:L1} and \eqref{e:chicritextension}.  Indeed, if
$\frac 1z=n+\tau_n n^{1/2}$ with
$|\tau_n| \le O(n^{q-1/2})$ for $q \in [\frac 12, \frac 23)$, then we obtain
\begin{equation}
    E_z^{(n)}(L) = n-1- (n+\tau_n n^{1/2})
    + (n+\tau_n n^{1/2})(\chi_{z}^{(n)})^{-1} \sim \alpha_{\tau_n}\sqrt{n}.
\end{equation}
Since $\alpha_\tau \sim -\tau$ as $\tau \to -\infty$, this is consistent with the
first case of \eqref{e:ELsup}, and since
$\alpha_\tau \sim \tau^{-1}$ as $\tau \to +\infty$, it is consistent with the
second case of \eqref{e:ELsub}.

The first case of \eqref{e:ELsub} and the last case of \eqref{e:ELsup} were proved in \cite[Theorem~1.5]{Yadi16},
as was the case $\tau=0$ of the critical case of \eqref{e:EL}
($\alpha_0 = \sqrt{2/\pi}$ agrees with the constant $\alpha$ in \cite{Yadi16}),
by using a formula
for the expected length in terms of a Poisson random variable rather than employing
\eqref{e:chigam} as we do.  A proof that $E_z^{(n)}(L) \asymp n^{1-q}$
in the second case of \eqref{e:ELsub} was announced in
\cite{ZGFDG18} (the scaling of $z$ on \cite[p.185701-3]{ZGFDG18} differs from ours
by a factor $n$).

The next theorem gives the asymptotic distribution of $L$ rescaled as suggested
by Corollary~\ref{cor:EL}.
In its statement,
$G_p$ is a geometric
random variable with parameter $p$,
$W_a$ is an exponential random variable with mean $a^{-1}$, and
$X_\tau$ is the random variable (shown to exist in the proof) with moment generating function
\begin{equation}
\label{e:mgfX}
    M_{X_\tau}(t)=
    e^{ -\tau t +t^2/2}\frac{{\rm erfc}(2^{-1/2}(\tau-t))}{{\rm erfc}(2^{-1/2}\tau)}.
\end{equation}
Differentiation confirms that the constant $\alpha_\tau$ in \eqref{e:EL} is equal
to $\alpha_\tau = EX_\tau = M_{X_\tau}'(0)$.

\begin{theorem}
\label{thm:L}
Let $\tau \in \R$, $a>0$, and $q \in (\frac 12, 1)$.
As $n \to \infty$, the rescaled length converges in distribution
as follows:
\begin{alignat}{2}
\label{e:Lsubconv}
     L &\Rightarrow
    G_{1-1/s} -1 \qquad & (z=\tfrac{1}{sn},\, s>1),
    \\
\label{e:Lnearsubconv}
     L/n^{1-q} &\Rightarrow W_a & (z = \tfrac{1}{n+an^q}),
    \\
\label{e:Lcritconv}
     L/n^{1/2} &\Rightarrow X_\tau & (z = \tfrac{1}{n+\tau \sqrt{n}}),
    \\
\label{e:Lnearsupconv}
    L/n^q &\Rightarrow a & (z= \tfrac{1}{n-an^q}),
    \\
\label{e:Lsupconv}
    L/n &\Rightarrow 1-s& (z=\tfrac{1}{sn},\, s\in (0,1)).
\end{alignat}
\end{theorem}

The limits \eqref{e:Lnearsupconv}--\eqref{e:Lsupconv} are special cases of the
more general statement that in the supercritical case, with $\ell = \Ex_z^{(n)}(L)$,
\begin{alignat}{2}
\label{e:Lsupconvcombined}
    L/\ell &\Rightarrow 1  \qquad& (\text{supercritical}).
\end{alignat}
We will prove the more general statement in the proof of Theorem~\ref{thm:L}.

The proofs of Theorems~\ref{thm:chi} and \ref{thm:L} are given in Section~\ref{sec:pf}.
We have stated the results above in a simple manner in order to highlight
the leading behaviour.  Full asymptotic expansions could be obtained by using
the expansions of \cite{NO19} for the incomplete gamma function, but our focus is
on the leading behaviour and we do not
pursue higher precision here.

\smallskip \noindent \emph{Note added:} After the first version of this paper was
posted on arXiv, the author was kindly informed by T.M.~Garoni that the forthcoming paper
\cite{DGGNZ19} contains independently obtained proofs of many of our results, including
Theorem~\ref{thm:L}.  In \cite{DGGNZ19}, the distribution of $X_\tau$ is identified as
that of $-\tau$ plus a standard normal random variable conditioned to exceed $\tau$.

\subsection{Critical behaviour}

\paragraph{Summary.}
The above results can be summarised by saying that there is a \emph{critical scaling window}
of $z$ values of the form $n^{-1}(1+O(n^{-1/2}))$ where the susceptibility is of order
$n^{1/2}$, the expected length is of order $n^{1/2}$, and the limiting length
has an interesting distribution.
Values $z=1/(sn)$ are \emph{subcritical}
when $s>1$ ($L$ remains bounded) and \emph{supercritical}
when $s<1$ ($L$ is of order $n$ and the self-avoiding walk has positive density).
Progression into the critical window from the subcritical side
occurs for $z = n^{-1}(1-n^{-p})$
with $p \in (0,\frac 12)$.  Similarly, progression out of the critical window
to the supercritical side
occurs for $z = n^{-1}(1+n^{-p})$
with $p \in (0,\frac 12)$.  Both progressions involve power-law  behaviour
for $L$.

\paragraph{Percolation on finite graphs.}
It is instructive to compare our results with the situation for percolation
on high-dimensional finite transitive graphs such as the complete graph,
the hypercube, or discrete tori.
Percolation on the complete graph---the Erd\H{o}s--R\'{e}nyi random graph---is of
course a fundamental example in probability theory and combinatorics
with an extensive literature, e.g., \cite{JLR00}.
In \cite{BCHSS05a}, it was shown that
the correct definition in the high-dimensional setting is to define the critical probability
$p_c$ as the solution to the equation $\chi(p_c)=\lambda V^{1/3}$, where
$\chi(p)$ is the expected cluster size of a fixed vertex when the bond
occupation probability is $p$,
$V$ is the number of vertices in the graph, and $\lambda>0$ is an arbitrary
fixed constant.\footnote{The utility and freedom
to vary $\lambda$ is discussed in \cite{BCHSS05a}.  There is no
notion of ``the'' critical value on a finite graph, and any value in the critical
window will do.}
For the hypercube and high-dimensional discrete tori,
it has been proved\footnote{In fact, some open questions remain concerning
the supercritical phase on tori; see \cite{HH17book}.} \cite{BCHSS05a,HH17book} that there
is a phase transition with critical scaling window consisting of $p$
values with $|p-p_c|$ of order $\Omega^{-1}V^{-1/3}$, where $\Omega$
is the degree ($\Omega=n$ and $V=2^n$ for the hypercube, $\Omega=2d$
and $V=r^d$ for a $d$-dimensional
torus of period $r$).  For the hypercube, $p_c$ has an asymptotic expansion to all
orders, $p_c \sim \sum_{j=0}^\infty a_j n^{-j}$, with rational coefficients $a_j$ that
are independent of $\lambda$ (in particular, $a_1=a_2=1$, $a_3=\frac 72$) \cite{HS05}.
We emphasise that the power $V^{1/3}$ is the correct power in the high-dimensional
setting, but will not in general be correct for low-dimensional graphs:
in \cite[Section~3.4.2]{BCHSS05a}
it is proposed
that $V^{1/3}$ corresponds more generally to $V^{(\delta-1)/(\delta+1)}$ where
$\delta$ is the critical exponent for the magnetisation (with mean-field value $\delta=2$).

\paragraph{Conjectured picture for self-avoiding walk.}
For the complete graph, $n$ occurs both as the degree $\Omega=n-1$ and as the number
of vertices $V=n$, but the degree and volume will play distinct roles in other graphs.
By analogy with the above picture for percolation, our results suggest that the
natural definition\footnote{This definition differs from the proposal
in \cite[Definition~1.2]{Yadi16}, which calls $z_n$ \emph{subcritical} if
$\limsup \chi^{(n)}_{z_n} <\infty$, \emph{supercritical} if $\liminf \chi^{(n)}_{z_n} =\infty$,
and \emph{critical} if $z_n/s$ is subcritical when $s>1$ and supercritical when $s<1$.
By \eqref{e:chisub}, \eqref{e:chicrit}, and
\eqref{e:chisup}, this prescribes $z=\frac 1n$ as both critical
and supercritical for the complete graph, and thus is insufficiently
discerning.}
of the critical value for self-avoiding walk on a high-dimensional
transitive graph with $V$ vertices is the value $z_c$ for which $\chi(z_c)= \lambda V^{1/2}$
with any fixed $\lambda>0$.\footnote{For the complete graph, since the function
$\alpha_\tau$ is a bijection of $\R$ onto $(0,\infty)$, by the critical case
of \eqref{e:EL} a choice of $\lambda$
essentially specifies a choice of $\tau_\lambda$ such that the susceptibility
at $z=1/(n+\tau_\lambda n^{1/2})$ has the value $\lambda n^{1/2}$.}
Also, our results suggest that the critical scaling
window consists of $z$ values with $|z-z_c|$ of order $\Omega^{-1} V^{-1/2}$, where $\Omega$
is the degree.
For the hypercube, the definition gives $z_c$ according to $\chi_n(z_c)=\lambda 2^{n/2}$
and the scaling window is conjectured to occur for $|z-z_c|$ of order $n^{-1}2^{-n/2}$.
For the discrete torus of period $r$ in dimension $d>4$, instead $\chi(z_c)=\lambda r^{d/2}$ with
a scaling window with $|z-z_c|$ of order $(2d)^{-1}r^{-d/2}$.
It would be of interest to prove that this conjectured picture does hold and
to carry out an analysis similar to what has been done for percolation.

\paragraph{Literature.}
The idea that for $z$ above a critical value the self-avoiding walk enters
a dense phase with walk length of the order of the volume has been explored in
previous papers.  A sharp transition to a dense phase for self-avoiding walks
diagonally crossing
a square or higher-dimensional cube was proved in \cite{Madr95a} and
further studied in \cite{BGJ05}.  The existence of a dense phase was established
for self-avoiding walk in a two-dimensional domain in \cite{DKY14}.  The vanishing of the
density as the critical point is approached from the supercritical side
was proven in \cite{GI95} for weakly self-avoiding
walk on a 4-dimensional hierarchical lattice,
with a logarithmic correction to the linear $(1-s)$ mean-field behaviour
evident in the last case of \eqref{e:ELsup}.  For finite graphs of large girth,
a dense phase was proven to exist in \cite{Yadi16} (where, as mentioned above,
the complete graph was also considered).  Forthcoming work on the complete graph
announced in \cite{ZGFDG18} subsequently appeared in \cite{DGGNZ19} (see Note Added
at end of Section~\ref{sec:results}).

\section{Proof of results}
\label{sec:pf}

Asymptotic formulas for the susceptibility are immediate
consequences of the following
asymptotic formulas for the incomplete gamma function, taken
from \cite{NO19}.
Asymptotic expansions to all orders are given in \cite{NO19},
but we only need the statements in Lemma~\ref{lem:G}.
The variable $y_n$ represents $1/z_n$.

\begin{lemma}
\label{lem:G}
Let $r > \frac 12$.
Suppose that eventually $y_n\ge n+ n^r$ (subcritical case).  Then
\begin{equation}
    \Gamma(n,y_n) = y_n^n e^{-y_n} \frac{1}{y_n-n}\left( 1- \frac{y_n}{(y_n-n)^2}
    + O\Big( \frac{y_n^2}{(y_n-n)^4} \Big)\right).
\end{equation}
Let $\tau_n,\tau\in\C$ with $\tau_n\to\tau$,
and let $y_n=n+\tau_n n^{1/2}$.  Then
\begin{equation}
\label{e:Gammacrit}
    \Gamma(n,y_n) \sim \Gamma(n)
    \frac{1}{2} {\rm erfc}(2^{-1/2}\tau).
\end{equation}
Suppose that eventually $y_n  \le n-n^r$ (supercritical case).  Then
\begin{equation}
    \Gamma(n,y_n) \sim
    \Gamma(n)
    .
\end{equation}
\end{lemma}

\begin{proof}
The subcritical case follows from \cite[(2.2)]{NO19},
together with the observation
in \cite[Appendix~A]{NO19} that \cite[(2.2)]{NO19} does hold for
$y_n \ge n+n^r$ when $r> \frac 12$.
The critical case follows from \cite[Theorem~1.1]{NO19} (see also
\cite[Proposition~1.1]{NO19}\footnote{This is Proposition~1.2 in the arXiv version \url{https://arxiv.org/pdf/1803.07841.pdf} of \cite{NO19}.}).
The supercritical case follows from \cite[(2.1)]{NO19},
again with the observation
in \cite[Appendix~A]{NO19} that \cite[(2.1)]{NO19} does hold under our hypothesis.
In more detail for the last case, it follows from \cite[(2.1)]{NO19}
and Stirling's formula that, with $\delta_n = n-y_n \ge n^r$,
\begin{align}
    \left|\frac{\Gamma(n,y_n)}{\Gamma(n)} - 1 \right|
    & \sim
    \frac{1}{n-y_n} y_n^ne^{-y_n} \frac{1}{(n-1)^{n-1}e^{-n+1}\sqrt{2\pi n}}
    \nonumber \\  & \sim
    \frac{1}{\sqrt{2\pi}} \frac{n^{1/2}}{\delta_n} (1-\delta_n/n)^n e^{\delta_n}
    \le
    \frac{1}{\sqrt{2\pi}\, n^{r-1/2}} ,
\end{align}
and the right-hand side goes to zero since $r> \frac 12$.
\end{proof}

\begin{proof}[Proof of Theorem~\ref{thm:chi}]
The asymptotic formulas \eqref{e:chisub1}--\eqref{e:chisup1} for the susceptibility
follow from
\eqref{e:chigam} and Lemma~\ref{lem:G}.
For \eqref{e:chicrit}, we also  use Stirling's formula and the elementary fact
that
$(1+\tau_n n^{-1/2})^{-n}e^{\tau_n n^{1/2}} \to e^{\frac 12 \tau^2 }$.
The extension \eqref{e:chicritextension} of the critical case is a consequence
of the fact that \eqref{e:Gammacrit} continues to hold in the extended setting,
by \cite[Proposition~1.1]{NO19}.
\end{proof}

It remains to prove Theorem~\ref{thm:L}.
The proof uses the fact that convergence in distribution is a consequence of
convergence of moment generating functions on an open interval containing $0$,
except for the critical case
\eqref{e:Lcritconv} where we use characteristic functions instead.

\begin{proof}[Proof of \eqref{e:Lsubconv}: $z=1/(sn)$ with $s>1$ (subcritical)]
Fix $s>1$ and choose $t_0>0$ to obey $se^{-t_0} =1$.  Let $t<t_0$; then $se^{-t}>1$.
Let $y=1/z$ and $y_t=ye^{-t}$.
By \eqref{e:mgf} and \eqref{e:chisub1}, the moment generating function of $L$ obeys
\begin{align}
    M_L(t) &
    \sim \frac{y_t}{y_t-n}\frac{y-n}{y}
    =
    e^{-t} \frac{s-1}{se^{-t}-1}
    = \frac{1- \frac{1}{s}}{1-\frac 1s e^t}.
\end{align}
The right-hand side is the moment generating function for $G_{1-1/s} -1$.
\end{proof}

\begin{proof}[Proof of \eqref{e:Lnearsubconv}: $z= 1/(n+an^q)$  (subcritical near critical)]
Fix $a>0$, $q \in (\frac 12, 1)$, and let $t<a$.
Let $y=1/z$ and $y_t = ye^{-t/n^{1-q}}$.
By definition, $y-n  = an^q$ and
\begin{align}
    y_t -n & \sim (n+an^q)(1-tn^{q-1}) -n \sim (a-t)n^q.
\end{align}
Since $t<a$, by \eqref{e:mgf} and \eqref{e:chisub1} the moment generating function $M_{L/n^{1-q}}(t)$ is
\begin{align}
    M_L(t/n^{1-q}) & \sim
    \frac{y_t}{y}\frac{y-n}{y_t-n}
    \sim
    e^{-t/n^{1-q}}\frac{a}{a-t} \to \frac{a}{a-t},
\end{align}
and the right-hand side is the moment generating function for $W_a$.
\end{proof}

\begin{proof}[Proof of \eqref{e:Lcritconv}: $z = 1/(n+\tau \sqrt{n})$ (critical window)]
Fix $\tau  \in \R$ and $t\in\C$.  Let $y=1/z$ and $y_t=ye^{-t/\sqrt{n}}$.
Since there is a sequence $x_n\to 0$ such that
\begin{equation}
    y_t= n + n^{1/2}(\tau-t+x_n),
\end{equation}
we see from the last member of \eqref{e:mgf} and \eqref{e:chicrit} that
\begin{align}
    M_{L/\sqrt{n}}(t) &
    \sim
    \frac{e^{(\tau-t)^2/2}{\rm erfc}(2^{-1/2}(\tau-t))}
    {e^{\tau^2/2}{\rm erfc}(2^{-1/2}\tau)},
\label{e:Mcrit}
\end{align}
which agrees with \eqref{e:mgfX}.
In particular, \eqref{e:Mcrit} holds for $t=i\theta$ (since \eqref{e:Gammacrit} does),
in which case
$M_{L/\sqrt{n}}(i\theta)$ is the characteristic function of $L/\sqrt{n}$.
Since the right-hand side
of \eqref{e:Mcrit} is a continuous function of $\theta$ in this case, the right-hand
side is the characteristic function of a random variable, and $L/\sqrt{n}$ converges
in distribution to this random variable.\footnote{The characteristic function is
useful for this last step, as general theory then provides the existence of the limiting
random variable.  With moment generating functions, this is less obvious.}
\end{proof}

\begin{proof}[Proof of \eqref{e:Lsupconvcombined}: ($z$ supercritical)]
Let $r> \frac 12$, $y=1/z$, and suppose that eventually $y \le n-n^r$.
Let $\ell = \Ex_z^{(n)}(L)$; then $\ell \sim n-y$ by \eqref{e:EL}.
Let $t \in \mathbb{R}$ and $y_t = ye^{-t/\ell}$.
Then $y_t = y+(y_t-y) \sim y - yt/\ell$.  Since eventually
\begin{align}
    \frac{y}{\ell} & \sim \frac{y}{n-y} \le \frac{y}{n^r} \le n^{1-r},
\end{align}
and since $n^{1-r} = o(n^r)$ because $r> \frac 12$,
eventually
$y_t \le n-n^r+o(n^r) \le n - n^{s}$ for some $s \in (\frac 12, r)$.
Therefore, by \eqref{e:mgf} and \eqref{e:chisup1},
\begin{align}
    M_{L/\ell}(t) = M_L(t/\ell) & \sim
    \frac{y_t^{1-n} e^{y_t}}{y^{1-n}e^y} =
    e^{t(n-1)/\ell} e^{y_t -y}.
\end{align}
The logarithm of the right-hand side is equal to
\begin{align}
    \frac{t(n-1)}{\ell} + y(e^{-t/\ell} - 1) & \sim
    \frac{t(n-1)}{\ell} - \frac{ty}{\ell}
    =
    \frac{t(n-y)}{\ell} - \frac{t}{\ell}
    \sim
    t - \frac t\ell \sim t,
\end{align}
and hence
\begin{align}
    M_L(t/\ell) & \to
    e^{t}.
\end{align}
The right-hand side is the moment generating function of the constant random variable $1$,
and the proof is complete.
\end{proof}

\section*{Acknowledgements}
This work was supported in part by NSERC of Canada, and was carried
out at the Research
Institute for Mathematical Sciences, Kyoto University.
I am grateful to Takashi Kumagai and Ryoki Fukushima
for support and hospitality in Kyoto,
to Tyler Helmuth for discussions and advice, and to Tim Garoni for informing me
of independent work on this topic with similar results which subsequently
appeared in \cite{DGGNZ19}.


\end{document}